\documentclass[12pt,a4paper]{amsart}
\usepackage{amssymb,amscd}
\usepackage{mathrsfs}
\usepackage[mathcal]{eucal}
\usepackage{float}
\usepackage[stable]{footmisc}

\usepackage{xy}
\input xy
\xyoption{all}
\usepackage{pdflscape}
\usepackage{hyperref}

\usepackage[english]{babel}

\def\B{\mathbb{B}}

\let\myacute=\'

\def\<{\langle}
\def\>{\rangle}
\def\N{\mathbb{N}}
\def\Z{\mathbb{Z}}

\def \begindm {\begin{displaymath}}

\def \enddm {\end{displaymath}}

\newtheorem{theo}{Theorem}[section]

\newtheorem{claim}{Claim}

\theoremstyle{definition}

\newtheorem{remark}[theo]{Remark}

\numberwithin{equation}{section}
\newtheorem{note}[theo]{Note}

\long\def\symbolfootnote[#1]#2{\begingroup\def\thefootnote{\fnsymbol{footnote}}\footnote[#1]{#2}\endgroup}

\title[Enrichments of Boolean algebras]{Enrichments of Boolean algebras: A Uniform Treatment of Some Classical and Some Novel Examples}

\author[J. Derakhshan]{Jamshid Derakhshan}
\address{University of Oxford, Mathematical Institute,
24-29 St Giles', Oxford OX1 3LB, UK}
\email{derakhsh@maths.ox.ac.uk}

\author[A. Macintyre]{Angus Macintyre}
\address{Queen Mary, University of London,
School of Mathematical Sciences, Queen Mary, University of London, Mile End Road, London E1 4NS, UK}
\email{angus@eecs.qmul.ac.uk}

\begin{document}



\begin{abstract} We give a unified treatment of the model theory of various enrichments of infinite atomic Boolean algebras, 
with special attention to quantifier-eliminations, complete axiomatizations and decidability. 
A classical example is the enrichment by a predicate for the ideal of finite sets, 
and a novel one involves predicates giving congruence conditions on the cardinality of finite sets. We focus on three 
examples, and classify them by expressive power.
\end{abstract}

\maketitle

\section{Introduction}

In the course of some work on the model theory of adeles \cite{DM1,DM2}, we needed, in connection with the use of 
Feferman-Vaught Theorems \cite{FV}, to appeal to various classical results about the model theory of enrichments of 
Boolean algebras. For some of these, published proofs were hard to find. Relevant information can be 
found in \cite[Chapter 2, Section 6]{KK},\cite{CK}, and \cite{Ershov}. Moreover, we became aware, motivated by the 
examples of the adeles, that some novel enrichments were interesting. It turned out their model theory, and that of the 
classical examples, could be given a simple common treatment. This is the content of the present paper.

\section{Enrichments of infinite atomic Boolean algebras}

{\bf Example 1.} $T_1$ is the theory of infinite atomic Boolean algebras, in the Boolean language with $0,1,\cap,\cup,\neg$. 
The main models are $Powerset(I)$ (which denotes the powerset of $I$), for $I$ infinite. These are clearly not the only models, since no 
countable model is a full powerset algebra. 
A complete set of axioms \cite{CK} is given 
by saying that our models are infinite Boolean algebras such that every nonzero element has an atom below it. 

For this and the examples below we prove quantifier elimination by a variant of the standard back-and-forth criterion, 
in a form given in Hodges' book \cite[Exercise 4,pp.389]{Hodges}, 
especially well suited to our situation. To apply this criterion it is crucial to note that Boolean algebras are locally 
finite. We shall demonstrate the forth-stages of the argument, since the back-stages are completely analogous. 

We find it slightly more enlightening to work in the equivalent formalism of Boolean rings \cite{Birkhoff}, using 
the dictionary 
$$x.y=x\cap y$$
$$x+y=(x\cap \neg y)\cup (\neg x\cap y)$$
$$x\cap y= x.y$$
$$x\cup y=x+y+xy$$
Since all these ``definitions'' are quantifier-free, we can prove our quantifier-elimination by working in the categories 
of enriched Boolean rings. An ordering on a Boolean ring is defined by $x\leq y$ if and only if $x.y=x$.

The predicates needed for the quantifier elimination in this case are $C_n(x),~n\geq 1$, with the interpretation that there are 
at least $n$ distinct atoms $\alpha$ with $\alpha\leq x$.

Now suppose $\B_1$ and $\B_2$ are $\omega$-saturated models of $T$, $\{\alpha_1,\dots,\alpha_m\}$, 
$\{\beta_1,\dots,\beta_m\}$ are finite Boolean subrings $R_1,R_2$ of $\B_1,\B_2$ respectively, and 
$$F(\alpha_j)=\beta_j$$ 
is an isomorphism of Boolean rings, in addition respecting all $C_n$ and $\neg C_n$ (interpreted respectively in $\B_1,\B_2$). Now in fact 
$m=2^k$ for some $k\geq 1$. $R_1$ and $R_2$ are each atomic, but their atoms need not be atoms of $\B_1,\B_2$. If $k=1$, 
$$R_1=\{0,1\}\subset \B_1,$$
$$R_2=\{0,1\}\subset \B_2.$$
Note that if some $\alpha_j$ is an atom of $\B_1$, then 
$$\B_1\models C_1(\alpha_j)\wedge \neg C_2(\alpha_j)$$
so
$$\B_2\models C_1(\beta_j)\wedge \neg C_2(\beta_j),$$
so $\beta_j$ is an atom of $\B_2$.

We use systematically the following function:

\begindm
\sharp(x)=\begin{cases}
n&\textrm{if}\ C_n(x)\wedge \neg C_{n+1}(x)\\
\infty&\textrm{if no such $n$ exists}.
\end{cases}
\enddm

Note that any map respecting each $C_n$ and $\neg C_n$ preserves $\sharp$.

Now we do the back-and-forth argument. Let $\alpha$ be an element of $\B_1$ not in $R_1$. We try to extend $F$ to the Boolean 
ring 
$$R_1[\alpha]=\{r_1+s_1.\alpha: r_1,s_1\in R_1\}$$ 
of cardinal between $2^k$ and $2^{2k}$.
%$2^{k+1}$ (it has cardinal greater than $2^{k}$ and at most 
%$2^{k+1}$, and a power of $2$.)

\

{\bf Note:} In any atomic Boolean algebra, every non-zero element is the supremum of the atoms 
below it (see \cite{Birkhoff}).

\

In particular, $R_1[\alpha]$ has atoms not in $R_1$. We get to $R_1[\alpha]$ from $R_1$ by successive 
adjunctions of the atoms of $R_1[\alpha]$, and so without loss of generality (for the extension problem) we can assume that 
$\alpha$ is an atom of $R_1[\alpha]$. We assume this henceforward.

\

{\bf Case 1}: $k=1$. 

$\alpha$ and $1-\alpha$ are atoms of $R_1[\alpha]$, though not necessarily of $\B_1$. Note that not both of 
$\sharp(\alpha)$ and $\sharp(1-\alpha)$ can be finite, but that and being nonzero is the only restriction on the 
pair $(\sharp(\alpha),\sharp(1-\alpha))$.

Clearly the extension problem is solved once one has a $\beta\in \B_2$ with $\beta \notin \{0,1\}$ and 
$$(\sharp(\beta),\sharp(1-\beta))=(\sharp(\alpha),\sharp(1-\alpha)).$$
If $\sharp(\alpha)$ is finite, it is trivial to get $\beta$ with $\sharp(\beta)=\sharp(\alpha)$ (just take $\beta$ a sum of 
$\sharp(\alpha)$ atoms), and then $\sharp(1-\beta)=\sharp(1-\alpha)$ automatically).

If $\sharp(1-\alpha)$ is finite, a dual argument works. If 
$$\sharp(\alpha)=\sharp(1-\alpha)=\infty,$$
we use $\omega$-saturation of  
$\B_2$ to get $\beta$ with 
$$\sharp(\beta)=\sharp(1-\beta)=\infty.$$

\

{\bf Case 2: $k>1$}.

Now $1=\gamma_1+\dots+\gamma_k$ where the $\gamma_i$'s are the atoms of $R_1$. It follows that 
for some $i_0\in \{1,\dots,k\}$ we must have 
$$0 < \alpha.\gamma_{i_0} < \gamma_{i_0}.$$
Indeed, if 
$\alpha.\gamma_i=\gamma_i$ for all the $\alpha.\gamma_i$ which are nonzero, where $i\in \{1,\dots,k\}$; then 
$$\alpha=\alpha.1=\alpha(\sum_{1\leq i\leq k} \gamma_i)=\sum_{1\leq i\leq k} \alpha.\gamma_i=
\sum_{1\leq i\leq k,\alpha.\gamma_i\neq 0} \gamma_i\in R_1,$$
a contradiction. 

Let $A=\alpha.\gamma_{i_0}$. Note that $A.\gamma_j=0$ for all $j\neq i_0$ in $\{1,\dots,k\}$ 
since $\gamma_j.\gamma_{i_0}=0$ for all $j\in \{1,\dots,k\}$. 
Since $\alpha$ is an atom of $R_1[\alpha]$, we must have $A=\alpha.\gamma_{i_0}=\alpha$.  So we have shown that $\alpha$ lies 
below a unique atom $\gamma_{i_0}$ of $R_1$. In the following we 
shall write $\gamma$ for $\gamma_{i_0}$. 

It follows that the atoms of $R_1[\alpha]$ are:

i) the atoms of $R_1$ distinct from $\gamma$;

ii) $\alpha$ and $\gamma-\alpha$.

Now an arbitrary element of $R_1[\alpha]$ can be represented uniquely in the form 
$$\epsilon_1.\alpha+\epsilon_2.(\gamma-\alpha)+\sum_{\tau}\epsilon_{\tau}.\tau$$
where the $\tau$-summation is over all atoms of $R_1$ different from $\gamma$, 
and the $\epsilon$'s are each $0$ or $1$. Note that the three summands are pairwise disjoint. 

Now clearly
$$\sharp(\epsilon_1.\alpha+\epsilon_2.(\gamma-\alpha)+\sum_{\tau}\epsilon_{\tau}.\tau)=
\epsilon_1\sharp(\alpha)+\epsilon_2\sharp(\gamma-\alpha)+\sum_{\tau}\epsilon_{\tau}\sharp(\tau)$$
So the extension problem this time is to find $\beta$ with
$$0<\beta<F(\gamma), ~\mathrm{and}$$
$$\sharp(\beta)=\sharp(\alpha),$$ 
$$\sharp(F(\gamma)-\beta)=\sharp(\gamma-\alpha).$$
Now the key issue is $\sharp(\gamma)$ ($=\sharp(F(\gamma))$).

\

{\bf Subcase 1}: $\sharp(\gamma)$ finite.

Then $\sharp(\gamma)=\sharp(\alpha)+\sharp(\gamma-\alpha)$, and both $\sharp(\alpha)$ and 
$\sharp(\gamma-\alpha)$ are greater than $0$. 

To solve the extension problem we simply choose $0<\beta<F(\gamma)$ 
with 
$$\sharp(\beta)=\sharp(\alpha),$$
and then it is automatic that 
$$\sharp(F(\gamma)-\beta)=\sharp(\gamma-\alpha).$$

\

{\bf Subcase 2}: $\sharp(\gamma)$ infinite.

Then (cf.\ the slightly different Case 1) not both $\sharp(\alpha)$ and $\sharp(\gamma-\alpha)$ can 
be finite, but there is no other constraint except that each is positive. 

The argument goes exactly as in Case 1, with an appeal to $\omega$-saturation when both 
$\sharp(\alpha)$ and $\sharp(\gamma-\alpha)$ are infinite.

We have proved the following. 
\begin{theo}\label{thm-1} The theory of infinite atomic Boolean algebras in the enriched Boolean 
language with all the $C_n$ is complete, decidable, has quantifier elimination, and is axiomatized by sentences 
saying that the models are infinite Boolean algebras and every nonzero element has an atom below it.
\end{theo}
\begin{note} We build on this example to get analogous results for several enriched formalisms. The essential point will 
 be that the choice of $\beta$ will now involve more constraints than in the above, and all our work will be to show these 
constraints can be met. 
\end{note}

\

{\bf Example 2.} We enrich the language of Example 1 by a unary predicate $Fin$, and extend the axioms of $T_1$ by axioms saying that 
$Fin$ is a proper ideal, and, for each $n<\omega$,
$$\forall x (\sharp(x)\leq n \Rightarrow Fin(x)).$$
We call these the {\it basic ideal axioms}. They can be stated for any ideal $J$ (in our case $J=Fin$). 

In addition we add the crucial

{\bf Main Axiom}: $\forall x (\neg Fin(x) \Rightarrow (\exists y)(y<x \wedge \neg Fin(y) \wedge \neg Fin(x-y)))$.

In this way we get a theory $T_2$. 
We interpret $Fin$ in $Powerset(I)$ as the ideal of {\it Finite} sets. Note that Theorem \ref{thm-non-def-fin} in Section \ref{sec-rel} 
shows that 
$Fin$ is not definable in the language of Example 1. So this is not a definitional expansion of the theory of infinite atomic 
Boolean algebras.

We will prove a quantifier elimination as in Example 1, using the $C_n$ and $Fin$. We use the same notation and formalism, except
that in addition the map $F$ now respects $Fin$ and $\neg Fin$. 

\

{\bf Case 1}: $k=1$.

We add $\alpha$ and $1-\alpha$ and want to extend $F$ to $R_1[\alpha]$. We already know how to handle the various 
possibilities for $\sharp(\alpha)$ and $\sharp(1-\alpha)$. 

Note that one can not have both $Fin(\alpha)$ and $Fin(1-\alpha)$. 

In the case $\sharp(\alpha)<\infty$, we have $Fin(\alpha)$ and $\neg Fin(1-\alpha)$. If we choose $\beta\notin \{0,1\}$ with 
$\sharp(\alpha)=\sharp(\beta)$ (as we can do by Example 1) it is automatic that $Fin(\beta)$ and 
$\neg Fin(1-\beta)$. 

Similarly 
if $\sharp(1-\alpha)<\infty$, we have $Fin(1-\alpha)$ and $\neg Fin(\alpha)$, and again choosing $\beta$ as in Example 1 gives 
$\sharp(1-\beta)=\sharp(1-\alpha)$, $Fin(1-\beta)$, and $\neg Fin(\beta)$. 

So the remaining case is 
$$\sharp(\alpha)=\sharp(1-\alpha)=\infty.$$
In this case more care is required as it leaves open the possibility that $Fin(\alpha)\wedge \neg Fin(1-\alpha)$, or 
$\neg Fin(\alpha)\wedge Fin(1-\alpha)$. (Can happen by compactness in a nonstandard model). 

\

{\bf Subcase 1}: $Fin(\alpha)$. 

We simply have to choose $\beta$ so that 
$$Fin(\beta),~\sharp(\beta)=\infty,$$
$$\sharp(1-\beta)=\infty,\ \neg Fin(1-\beta).$$
This is trivial by $\omega$-saturation.

\

{\bf Subcase 2}: $\neg Fin(\alpha)$.

There are two subcases: 

\

{\bf Subsubcase 2.1}: $Fin(1-\alpha)$

We have to use $\omega$-saturation {\it and} the Main Axiom. First use the Main Axiom to find some $\delta$ in $\B_2$ 
with
$$\neg Fin(\delta),\ \neg Fin(1-\delta).$$
Now use $\omega$-saturation to get $\mu$ with
$$\mu\leq 1-\delta,\ Fin(\mu),\ \sharp(\mu)=\infty.$$
Now take $\beta$ as $1-\mu$. Clearly $\neg Fin(\beta)$ but $Fin(1-\beta)$, and $
\sharp(1-\beta)=\infty$.

\

{\bf Subsubcase 2.2}: $\neg Fin(1-\alpha)$.

Just use Main Axiom to get $\beta$ with $\neg Fin(\beta)$ and $\neg Fin(1-\beta)$. 

Now we get to 

\

{\bf Case 2}: $k>1$. 

As before we need only make minor changes to the procedure in Example 1. We preserve the 
notation (especially for the atom $\gamma$), and try to extend $F$ to preserve $Fin$ (and $\neg Fin$) as well. 
As in Example 1, we can assume that $\alpha$ is an atom of $R_1[\alpha]$.

So we have 
$$Fin(\epsilon_1.\alpha+\epsilon_2.(\gamma-\alpha)+\sum_{\tau} \epsilon_{\tau}.\tau)\Leftrightarrow 
Fin(\epsilon_1.\alpha)\wedge Fin(\epsilon_2.(\gamma-\alpha)\wedge Fin(\sum_{\tau}\epsilon_{\tau}.\tau),$$
where the $\epsilon$'s are either $0$ or $1$. Note that the summands are disjoint.

Thus it is clear that $Fin$ and $\neg Fin$ are preserved by the choice of $\beta$ if and only if 
$$Fin(\alpha)\Leftrightarrow Fin(\beta)$$
and
$$Fin(\gamma-\alpha)\Leftrightarrow Fin(F(\gamma)-\beta)$$
(provided $0<\beta<F(\gamma)$).

If $Fin(\gamma)$ then clearly $Fin(\alpha)$ and $Fin(\gamma-\alpha)$, with the same for $\beta$ and $F(\gamma)-\beta$ if 
chosen as in Example 1. 

If $\neg Fin(\gamma)$ then at least one of $\alpha$ and $\gamma-\alpha$ satisfies $\neg Fin$, with no other constraint 
except that $\sharp(\alpha)$ and $\sharp(\gamma-\alpha)$ are each nonzero.

\

{\bf Subcase 1}: $Fin(\alpha)$ and $\sharp(\alpha)<\infty$.

This is handled just as in Example 1.

\

{\bf Subcase 2}: $Fin(\alpha)$ and $\sharp(\alpha)=\infty$. 

Then automatically $\neg Fin(\gamma-\alpha)$. So we need
$$\beta\leq F(\gamma),\ Fin(\beta),\ \sharp(\beta)=\infty.$$
This is easily done by $\omega$-saturation.

\

{\bf Subcase 3}: $Fin(\gamma-\alpha)$ and $\sharp(\gamma-\alpha)<\infty$.

Exactly like Subcase 1.

\

{\bf Subcase 4}: $Fin(\gamma-\alpha)$ and $\sharp(\gamma-\alpha)=\infty$. 

Exactly like Subcase 2. 

\

{\bf Subcase 5}: $\neg Fin(\alpha)$ and $\neg Fin(\gamma-\alpha)$. 

%This is like the case $\sharp(\alpha)=\sharp(1-\alpha)=\infty$ 
%in Example 1, and it is handled using $\omega$-saturation.
By Main Axiom applied below $F(\gamma)$, there exists $\beta \in \B_2$ such that 
$\neg Fin(\beta)$ and $\neg Fin(F(\gamma)-\beta)$.

This concludes the proof of quantifier-elimination in Example 2. We have proved the following 
\begin{theo}\label{thm-2} The theory of infinite atomic Boolean algebras with the 
set of finite sets distinguished is complete, decidable and 
has quantifier elimination with respect to all the $C_n$ and $Fin$. The axioms required for completeness are the 
axioms of $T_2$ together with 
sentences expressing that $Fin$ is a proper ideal, the sentence
$$
\forall x (\neg Fin(x) \Rightarrow (\exists y)(y<x \wedge \neg Fin(y) \wedge \neg Fin(x-y))).
$$
and, for each $n<\omega$, the sentence $\forall x (\sharp(x)\leq n \Rightarrow Fin(x))$.
\end{theo}

\begin{remark} Note that $T_2$ is not complete if we remove the Main Axiom since in that case the finite-cofinite 
algebra on an index set $I$ (defined as the set of finite and cofinite subsets of $I$, and denoted $\B_{fin/cofin}(I)$) 
and the powerset $Powerset(I)$ are both models which are not elementarily equivalent. 

Note that in the Boolean language with $\{0,1,\cap,\cup,\neg\}$, $\B_{fin/cofin}(I)$ is an elementary substructure of 
$Powerset(I)$. This follows from Theorem \ref{thm-1} since $\B_{fin/cofin}(I)$ and $Powerset(I)$ have the same atoms.
\end{remark}
\begin{note} There are many complete extensions of the basic ideal axioms (for an ideal $J$). 
The Main Axiom gives a unique one, as does 
the axiom $\B/J\cong \{0,1\}$ (true in the finite-cofinite algebra). There are also examples where 
$\B/J\cong \B_k$, where $\B_k$ is a fixed finite Boolean algebra. 

A construction of such a Boolean algebra can be given as follows. 
Let $\B=\B_k^{\omega}$, where $\B_k$ is a $k$-element Boolean algebra, i.e.\ the functions $f:\omega \rightarrow \B_k$. 
Note that $\B$ is atomic with atoms the functions which are $0$ except at one $n\in \omega$, where the value is an atom. 
Let 
$$J=\{f\in \B: f(0)=0\}.$$
Then $\B/J\cong \B_k$.
\end{note}

{\bf Example 3.} 
This is built on top of Example 2, and seems to be novel. It is a kind of hybrid of Presburger arithmetic and the 
preceding example.
Example 1 is classical, done in \cite[Theorem 16,pp.70]{KK}, and Example 2 is classical, mentioned in \cite{FV}. Before 
presenting Example 3 we remark that it may be possible to find strengthenings of this example using well-behaved 
strengthenings of Presburger arithmetic (see \cite{Michaux}).

We add to the language of Example 2 
unary predicates $Res(n,r)(x)$ for $n,r\in \Z$, $n>0$, with the intended interpretation, in $Powerset(I)$, 
that $Fin(x)$ and the cardinal of $x$ is congruent to $r$ modulo $n$. There are various ``arithmetic'' axioms 
aside from the
$$\forall x (Res(n,r)(x)\Rightarrow Fin(x)),$$
for all $n,r$. For example, one clearly wants an axiom scheme stating that if $Fin(x)$ holds and $\sharp(x)=m$ where $m$ 
is congruent to $r$ modulo $n$, 
then $Res(n,r)(x)$ holds. Note that this implies 
$$Res(n,0)(0).$$
Also, we need
$$\forall x(Res(n,r)(x) \wedge r\equiv s (\mathrm{mod}~n) \Rightarrow Res(n,s)(x)),$$
and
$$\forall x(Res(n,r)(x)\wedge r \not\equiv s (\mathrm{mod}~n) \Rightarrow \neg Res(n,s)(x),$$
for all $n,r,s$. One also needs
$$\forall x(Res(m,r)(x)\Rightarrow Res(n,r)(x)),$$
if $n|m$, and 
$$\forall x(Fin(x)\Rightarrow \bigvee_{0\leq r< n} Res(n,r)(x)),$$
for all $m,n$.

Finally, we need ``finite additivity'' axioms, namely:
%$$\forall x \forall y (Res(n,r)(x) \wedge Res(n,s)(y) \Rightarrow 
%\bigvee_{\substack{0\leq j<n\\0\leq k<n\\k-j\equiv r+s (\mathrm{mod}~n)}}} (Res(n,k)(x\cup y)\wedge Res(n,j)(x\cap y))).$$
%And
%$$\forall x (Fin(x) \Rightarrow \exists !r~(0\leq r<n \wedge Res(n,r)(x))).$$
$$\forall x \forall y(x\cap y=0 \wedge Res(n,r)(x) \wedge Res(n,s)(y) \Rightarrow Res(n,r+s)(x\cup y)),$$
for all $n,r,s$; and
$$\forall x \forall y(x\cap y=0 \wedge Res(n,r)(x\cup y) \Rightarrow \bigvee_{\substack{0\leq s<n\\
0\leq t<n\\s+t\equiv r (\mathrm{mod}~n)}} Res(n,s)(x) \wedge Res(n,t)(y)),$$
for all $n,r$.

[It is easy to deduce from this the extension to the case of more than two variables, in inclusion/exclusion style.]

We call these the Boolean-Presburger axioms. Adding them to the axioms of $T_2$ we get a theory $T_3$. 
Now we try to elaborate the back-and-forth of Example 2, with the initial assumption that 
$F$ on $R_1$ respects all the $C_n,~Fin$, and all $Res(n,r)$. 

\

{\bf Case 1}: $k=1$. 

If neither $Fin(\alpha)$ nor $Fin(1-\alpha)$ there is nothing to prove, as all 
$Res(n,r)(\alpha)$ and $Res(n,r)(1-\alpha)$ are false, and the same will be true for the matching $\beta$ used in Example 2.

If (exactly) one satisfies $Fin$, say $\alpha$, we consider two subcases.

\

{\bf Subcase 1}: $\sharp(\alpha)<\infty$.

In this subcase, the truth of $Res(n,r)(\alpha)$ is determined by whether
$$\sharp(\alpha) \equiv r (\mathrm{mod}~ n).$$
This transfers automatically to the matching $\beta$ of Example 2. 

\

{\bf Subcase 2}: $\sharp(\alpha)=\infty$.

Note that any condition $\neg Res(n,r)(\alpha)$ is equivalent to a finite disjunction of various 
$Res(n,s)(\alpha)$, and so by saturation we need only get, for any $m\geq 1$, a matching $\beta_{\Sigma,m}$ satisfying,
$$\sharp(\beta_{\Sigma,m})\geq m$$
and
$$Res(n,r)(\beta_{\Sigma,m}),\ (n,r)\in \Sigma,$$
for any finitely many conditions $Res(n,r)(x)$, where $(n,r)\in \Sigma$ (where $\Sigma$ is a finite set), 
satisfied by $\alpha$. 

The argument needed is a slight variant of that used in the corresponding case of Example 2 (which depends on a similar 
argument in Example 1). All we need is $\sharp(\beta_{\Sigma,m})\geq m$ and $\sharp(\beta_{\Sigma,m})$ 
in the nonempty set (of nonnegative integers)
$$\{l: l\equiv r (\mathrm{mod}~n),\ (n,r)\in \Sigma\}.$$
Here all we need is that any Presburger definable nonempty set of the form 
$$\{l: l\equiv r (\mathrm{mod}~n),\ (n,r)\in \Sigma\}$$
has arbitrarily large members. This is obvious. 

This, with saturation, gives the required $\beta$.
%then as usual $\omega$-saturation 
%gives us a $\beta$ satisfying
%$\sharp(\alpha)=\sharp(\beta)$ and $Res(n,r)(\beta)$ 
%for all $n,r$ as above. Note, here we use $\sharp(\alpha) \equiv r~(n)$.

\

{\bf Case 2}: $k>1$. 

Again we preserve the notation of Example 2 (so $\gamma$ is an atom of $R_1$ and 
$0<\alpha<\gamma$). 

We do the usual argument representing an arbitrary element of $R_1[\alpha]$ as a (disjoint) sum 
$$\epsilon_1.\alpha+\epsilon_2.(\gamma-\alpha)+\sum_{\delta}\epsilon_{\delta}.\delta$$
(see Examples 1 and 2).

By the disjointness, we see that just as $Fin$ (and $\neg Fin$) for such an element is determined by 
$Fin(\alpha)$ and $Fin(\gamma-\alpha)$, it is clear that then $Res(n,r)$ is determined by the $Res(n,r)(\alpha)$ 
and $Res(n,r)(\gamma-\alpha)$.

So a choice of $\beta$ will preserve the basic relations and functions if and only if 
$$Fin(\alpha) \Leftrightarrow Fin(\beta),$$
and 
$$Fin(\gamma-\alpha) \Leftrightarrow Fin(F(\gamma)-\beta),$$
(provided $0<\beta<F(\gamma)$),
and 
$$Res(n,r)(\alpha) \Leftrightarrow Res(n,r)(\beta)$$
and
$$Res(n,r)(\gamma-\alpha) \Leftrightarrow Res(n,r)(F(\gamma)-\beta).$$
%So we see that we need elaborate on Example 2 by showing that we can find $\beta$ so that 
%$$Fin(\beta),~Fin(\gamma-\beta),$$
%and all 
%$$Res(n,r)(\beta),~Res(n,r)(F(\gamma)-\beta)$$
%match up with the conditions for $\alpha$ (we already have the condition $0<\beta<F(\gamma)$). 
%This is routine, using $\omega$-saturation.
We first consider the case when $Fin(\gamma)$. Then 
clearly $Fin(\alpha)$ and $Fin(\gamma-\alpha)$, with the same for $\beta$ and 
$F(\gamma)-\beta$ if chosen as in Example 1 (where there are subcases). But what about 
$Res(n,r)(\beta)$, which must match $Res(n,r)(\alpha)$?
%Note that $Res(n,r)(\gamma-\alpha)$ will be determined by $Res(n,r)(\alpha)$, and the same for $\beta$.
We have to go back and look at the subcases:

\

{\bf Subcase 1:} $\sharp(\gamma)$ is finite.

As in Example 1 it is necessary to choose $0<\beta<F(\gamma)$ with $\sharp(\beta)=\sharp(\alpha)$.  It is then 
automatic that $Res(n,r)(\beta)$ matches $Res(n,r)(\alpha)$. 

\

{\bf Subcase 2} $\sharp(\gamma)$ is infinite (and $Fin(\gamma)$).

Then (cf.\ Case 2 in Example 1) not both $\sharp(\alpha)$ and $\sharp(\gamma-\alpha)$ can be finite, but there is no 
other constraint except that each is positive. 

If $\sharp(\alpha)$ is finite, then $Res(n,r)(\alpha)$ is determined by $\sharp(\alpha)$, and  
$Res(n,r)(\gamma-\alpha)$ is determined by disjointness. So in this case we need only match
$$\sharp(\beta)=\sharp(\alpha),$$ as in Example 1.

The case that $\sharp(\gamma-\alpha)$ is finite is dual.

The crucial case is when $\sharp(\alpha)$ and $\sharp(\gamma-\alpha)$ are both infinite. The matching problem is to get 
$\beta < F(\gamma)$ satisfying
$$\sharp(\beta)\geq m_1,$$
for all $m_1\in \N$, and 
$$\sharp(F(\gamma)-\beta)\geq m_2,$$
for all $m_2\in \N$, and
$$Res(n,r)(\alpha) \Rightarrow Res(n,r)(\beta),$$
for all $n,r \in \N$.

This is like Case 1, Subcase 2. The saturation argument follows as before by the argument about Presburger definable sets. 
%Then (cf. Example 1) not both $\sharp(\alpha)$ and $\sharp(\gamma-\alpha)$ can be finite, but 
%there is no other constraint except that each is positive. Knowing all $Res(n,r)(\alpha)$ 
%(and so all $Res(n,r)(F(\gamma))$), it suffices to match up $Res(n,r)(\alpha)$ and 
%$Res(n,r)(\beta)$. This can be done using saturation more or less as in Case 1. 

Next we have to consider the situation when $\neg Fin(\gamma)$ holds. Then at least one of 
$\alpha$ and $\gamma-\alpha$ satisfies $\neg Fin$, with no other constraint except that 
$\sharp(\alpha)$ and $\sharp(\gamma-\alpha)$ are each nonzero.
%The argument from Subcase 5 (Example 2) 
%works here too (when $\neg Fin(\alpha)$ and $\neg Fin(\gamma-\alpha)$, so $\neg Res(n,r)(\alpha)$ 
%and $\neg Res(n,r)(\beta)$).

Note that If $\neg Fin(\alpha)$ and $\neg Fin(\gamma-\alpha)$ both hold then 
the only way to have 
$$Fin(\epsilon_1.\alpha+\epsilon_2.(\gamma-\alpha)+\sum_{\delta} \epsilon_{\delta}.\delta)$$
is that 
$$\epsilon_1=\epsilon_2=0,$$
thus in the case $\neg Fin(\alpha)$ and $\neg Fin(\gamma-\alpha)$ both hold, the only elements of 
$R_1[\alpha]$ satisfying $Fin$ are in $R_1$, and so the $Res(n,r)$ are determined. So one just has to 
get $\beta$ with $\neg Fin(\beta)$ and $\neg Fin(F(\gamma)-\beta)$ as in Subcase 5 in Example 2. 

We go quickly through the other cases.

\

{\bf Subcase 1}: $Fin(\alpha)$ and $\sharp(\alpha)<\infty$.

Then $\sharp(\alpha)$ determines all $Res(n,r)(\alpha)$ and the choice of $\beta$ as in Example 2 gives the required 
correspondence.

\

{\bf Subcase 2}: $Fin(\alpha)$ and $\sharp(\alpha)=\infty$.

This is easily done by saturation, in the style of Subcase 2 of Case 1.

\

{\bf Subcase 3}: $Fin(\gamma-\alpha)$ and $\sharp(\gamma-\alpha)<\infty$.

Dual to Subcase 1.

\

{\bf Subcase 4}: $Fin(\gamma-\alpha)$ and $\sharp(\gamma-\alpha)=\infty$.

Dual to Subcase 2.

This concludes the proof, and we have shown the following.

\begin{theo}\label{thm-3} The theory of infinite atomic 
Boolean algebras in the enriched language with all the $C_n, Fin$, and all $Res(r,n)$, is complete, decidable, and has 
quantifier elimination. The axioms needed to get the elimination are the axioms of $T_2$ together with the Boolean-Presburger axioms 
as follows: 
$$\forall x (Res(n,r)(x)\Rightarrow Fin(x)),$$
$$\forall x (Fin(x) \wedge \sharp(x)=m \wedge m\equiv r (\mathrm{mod}~n) \Rightarrow Res(n,r)(x)),$$
$$\forall x(Res(n,r)(x) \wedge r\equiv s (\mathrm{mod}~n) \Rightarrow Res(n,s)(x)),$$
$$\forall x(Res(n,r)(x)\wedge r \not\equiv s (\mathrm{mod}~n) \Rightarrow \neg Res(n,s)(x),$$ 
$$\forall x(Res(m,r)(x)\wedge n|m \Rightarrow Res(n,r)(x)),$$
$$\forall x(Fin(x)\Rightarrow \bigvee_{0\leq r< n} Res(n,r)(x)),$$
for all $n,r,s,m$,
$$\forall x \forall y(x\cap y=0 \wedge Res(n,r)(x) \wedge Res(n,s)(y) \Rightarrow Res(n,r+s)(x\cup y)),$$
for all $n,r,s$; and
$$\forall x \forall y(x\cap y=0 \wedge Res(n,r)(x\cup y) \Rightarrow \bigvee_{\substack{0\leq s<n\\
0\leq t<n\\s+t\equiv r (\mathrm{mod}~n)}} Res(n,s)(x) \wedge Res(n,t)(y)),$$
for all $n,r$.
\end{theo}

\section{Relative Strength of the Three Formalisms}\label{sec-rel}

Note that for each example we have given a complete set of axioms in the appropriate formalism. Now we show that 
each example is more expressive than its predecessor.

\begin{theo}\label{thm-non-def-fin} In no model of the theory of infinite atomic Boolean algebras can we 
define in the formalism of Example 1, a predicate $Fin$ satisfying the axioms given in Example 2.
\end{theo}

\begin{proof} It suffices to show this for $Powerset(\omega)$. Suppose $\Phi(x)$ defines the intended interpretation of $Fin$ in 
$Powerset(\omega)$. $\Phi(x)$ can be taken as a Boolean combination of conditions 
$$p(x)=0,~C_k(q(v)),$$
for $k\leq N \in \N$, where $p(x)$ and $q(x)$ are Boolean ring polynomials. 

Going to disjunctive normal form 
we see that $\Phi(x)$ can be taken as a finite disjunction of conditions
$$p_1(x)=0 \wedge \dots \wedge p_k(x)=0\wedge p_{k+1}(x)\neq 0\wedge \dots \wedge p_{k+l}(x)\neq 0 \wedge$$
$$C_{s_1}(t_1(x)) \wedge \dots \wedge C_{s_d}(t_d(x))\wedge \neg C_{k_1}(r_{1}(x))\wedge \dots \wedge \neg 
C_{k_m}(r_{m}(x)).$$
Note that $k_i,s_j\geq 1$ and all the polynomials occurring are of one of the forms 
$$0,\ 1,\ 1+x,\ x.$$
Note that $C_k(0)$ is false and $C_k(1)$ is true. Also 
$$1+x=0\Leftrightarrow x=1.$$
So our conjunction can be taken as Boolean combination of 
$$x=1,\ x=0,$$
$$C_l(x),\ C_m(1+x).$$
Only finitely many $l,m$ occur in the disjunctive normal form. 

We need only consider conjunctions of the form
$$C_l(x) \wedge C_m(1+x) \wedge \neg C_r(x) \wedge \neg C_s(1+x),$$
where not each of $l,m,r,s$ need occur. Note only one of $r,s$ can occur.

Consider each conjunction separately. Those which contain some $\neg C_r(x)$ can define only the set of elements 
$a$ with $\sharp(a)<r$. So we need consider only conjunctions which contain no $\neg C_r(x)$.  

If in such a conjunction some $\neg C_s(1+x)$ occurs, the conjunction can define only sets $a$ with 
$\sharp(1+a)<s$, in particular only elements with $\sharp(a)=\infty$.

So we need only consider conjunctions
$$C_l(x)\wedge C_m(1+x),$$
where one of $C_l,C_m$ may be missing. Any $a\in Powerset(\omega)$ with 
$\neg Fin(a)\wedge \neg Fin(1+a)$ will not satisfy this, 
contradiction.
\end{proof}

The case of Example 3 is harder. We will show the following.

\begin{theo}\label{thm2} If $p$ is a prime, the predicates $Res(p,r)(x)$ are not definable in 
 $Powerset(\omega)$ from $Fin,~C_k$ and any $Res(q^m,r)$ for primes $q\neq p$.
\end{theo}

{\bf Note:} $Res(p,r)$ is definable from the $Res(p^k,s)$ where $s \equiv r~\mathrm{mod}~p$, for any $k>1$.

\begin{proof} We give the proof for $p=2$ and no other primes, and explain at the end the general method. 
We work in a nonstandard model $\B$ of $Th(Powerset(\omega))$ in the formalism of Example 3. What we need in this 
model is a $b\in Fin$ such that $C_k(b)$ for all $k\in \N$, and $Res(2,0)(b)$. For example, use 
compactness, or an ultrapower of $Powerset(\omega)$. 

As usual we replace $\B$ by the corresponding Boolean ring $R$. Think of $b$ as a nonstandard finite even element. In $R$ 
we have the ideal $Fin$, and the filter $Cofin$ (i.e. the $c$ such that $Fin(1+c)$). However we need to consider also 
$$Cofin^{ST}=\{c: \neg C_k(1+c)~\mathrm{for~some}~k\},$$
i.e. the ``standard'' cofinite sets.

\begin{claim} $\{b\}\cup Cofin^{ST}$ has the finite intersection property.

\end{claim}
\begin{proof} Clearly $Cofin^{ST}$ has the finite intersection property, and if 
$$b.\tau=0$$
for some $\tau\in Cofin^{ST}$, we have 
$$b\leq 1+\tau,$$
so $b$ is standard finite.
\end{proof}

So we get a non-principal ultrafilter $D$ containing $b$. We now extend $R$ by an element $\gamma$, with the conditions 
$$\gamma.\alpha=\gamma,~\mathrm{if}~\alpha\in D,$$
$$\gamma.\alpha=0,~\mathrm{if}~\alpha \notin D.$$
We use compactness to show that the conditions on 
$\gamma$ are finitely satisfiable in $R$.

If we have finitely many conditions 
$$\gamma.\alpha_1=\gamma,\dots,\gamma.\alpha_r=\gamma,$$
$$\gamma.\beta_1=0,\dots,\gamma.\beta_s=0$$
with $\alpha_1,\dots,\alpha_r\in D, \ \beta_1,\dots,\beta_s \notin D$, then 
$$\alpha_1\cap \dots \cap \alpha_r \in D,$$
and is infinite, and 
$$\beta_1 \cup \dots \cup \beta_r \notin D.$$
Therefore there exists
$$\delta \in (\alpha_1\cap \dots \cap \alpha_r) \setminus (\beta_1 \cup \dots \cup \beta_s),$$
since otherwise 
$$(\alpha_1\cap \dots \cap \alpha_r) \subset (\beta_1 \cup \dots \cup \beta_s),$$
which contradicts $D$ being an ultrafilter. This proves finite satisfiability. Note that the argument shows
there are in fact infinitely many such $\delta$. Thus we get the extension $R[\gamma]$. 

Note that $\gamma \notin R$: if $\gamma\in R$, then there is an atom $s$ of $R$ below $\gamma$. Since 
$s$ is a standard finite set,
$$1+s\in Cofin^{ST},$$
so $1+s\in D$, hence $s\notin D$, thus $\gamma.s=0$. 

Now $R[\gamma]=\{r+s.\gamma: r,s\in R\}$.
\begin{claim} $\gamma$ is an atom in $R[\gamma]$.
 
\end{claim}
\begin{proof} Assume that $(r+s.\gamma).\gamma=r+s.\gamma$. Then $r.\gamma+s.\gamma=r+s.\gamma$, so 
$r.\gamma=r$. 

If $r\in D$, then $r.\gamma=\gamma$, hence $\gamma=r\in R$, contradiction. Hence $r\notin D$, so 
$0=r.\gamma=r$, so $r=0$. Now $s.\gamma=\gamma$ or $s.\gamma=0$. So $\gamma$ is an atom of 
$R[\gamma]$.
%Suppose $r=0$. If $s\in D$, then $s.\gamma=\gamma$, so $r+s.\gamma=\gamma$. 
%If $s\notin D$, then $s.\gamma=0$, so $r+s.\gamma=0$. 
%If $r=\gamma$, then depending on whether $s$ is or is not in $D$, $r+s.\gamma$ will be either 
%$2\gamma=0$ or $\gamma$.
\end{proof}

\begin{claim} All the atoms of $R$ are also atoms of $R[\gamma]$.
 \end{claim}
\begin{proof} Let $t$ be an atom of $R$. Assume that $(r+s.\gamma).t=r+s.\gamma$. 
Then since $\gamma$ is an atom (which is by assumption different from $t$), we have 
$$r.t=r+s.\gamma.$$
We have two cases. First case is when $r\in D$, hence $r.t=0$. In this case $r+s.\gamma=0$. The second case is when 
$r\notin D$, hence $r.t=t$. In this case we have $t=r+s.\gamma$.
\end{proof}

\begin{claim} Any atom of $R[\gamma]$ is either an atom of $R$ or $\gamma$.
 
\end{claim}
\begin{proof} Suppose $r+s.\gamma$ is an atom of $R[\gamma]$. Since $\gamma$ is an atom 
 $s.\gamma=0$ or $s.\gamma=\gamma$. So $r+s.\gamma$ is either $r$ or $r+\gamma$. 

In the former case, it is an old atom. In the latter, $r+\gamma$ must be an atom. But this is a contradiction since
$\gamma.(r+\gamma)=\gamma.r+\gamma=\gamma$. Hence $r+\gamma=\gamma$ since $\gamma \neq 0$ and $r+\gamma$ is an 
atom.
%0<\gamma<r+\gamma$.
\end{proof}

\begin{claim} $R[\gamma]$ is atomic.
\end{claim}
\begin{proof} Consider $r+\gamma \in R[\gamma]$. Let $s$ be an atom of $R$ below $r$. Then 
$$1+s\in Cofin^{ST} \subset D,$$
so $s.\gamma=0$. Hence $s.(r+\gamma)=s$, 
thus $s\leq r+\gamma$.
\end{proof}

Now we can finish the proof of the Theorem. $R[\gamma]$ is infinite atomic and a model of the axioms 
of Examples 1 and 2. The predicates $C_k$ thus have an interpretation in $R[\gamma]$.

Note that for every element $a$ of $R$ satisfying $Fin$, $a\notin D$, hence $\gamma.a=0$, so 
$$\gamma=\gamma.1=\gamma(a+(1-a))=\gamma.(1-a),$$
so $\gamma$ lies below every cofinite element of $R$. 

We give an interpretation $Fin^*$ of the predicate $Fin$ in $R[\gamma]$ as
$$Fin^*(r+\gamma)\Leftrightarrow Fin(r).$$
$$Fin^*(r)\Leftrightarrow Fin(r),~\mathrm{for}~r\in R$$
Note that $Fin^*$ is closed under addition since $Fin$ is so. Moreover 
$$r.(s+\gamma)=r.s+\gamma,$$
$$(r+\gamma)(s+\gamma)=r.s+(r+s).\gamma+\gamma=r.s+\gamma,$$
hence $Fin^*$ is closed under multiplication and is an ideal in $R[\gamma]$ since $Fin$ is so.

As for the predicates $C_k$ we define
$$C_k(r+\gamma)\Leftrightarrow C_k(r)$$
if $r\notin D$, and 
$$C_k(r+\gamma)\Leftrightarrow C_{k-1}(r)$$
if $r\in D$.

We need to show that the predicates $Fin,C_k$ and $\neg C_k$ are preserved in passing from $R$ to $R[\gamma]$.
This is clear for $Fin$. If $r\in R$ satisfies $C_k$ in $R$, then it clearly satisfies $C_k$ in $R[\gamma]$. 
Suppose $\neg C_l(r)$ holds in $R$. Then $r$ is a standard finite element. So
$$1+r\in Cofin^{ST}\subset D,$$
hence $\gamma(1+r)=\gamma$, so $\gamma.r=0$, therefore $\gamma$ does not lie below $r$, thus
$$R[\gamma]\models \neg C_l(r).$$
Since $\gamma$ is a new atom below any element $r\in D$ and $b\in D$, 
$b$ is not an even element in $R[\gamma]$, and 
$$R[\gamma] \nvDash Res(2,0)(b).$$
This proves that $Res(2,0)$ is not definable from $C_k$ and $Fin$.

In general, we give the predicates 
$Res(n,j)$ interpretations in $Fin^*$ as follows:
$$Res^*(n,j)(r)=Res(n,j+1)(r)~\mathrm{if}~r.\gamma=\gamma,$$
$$Res^*(n,j)(r)=Res(n,j)(r)~\mathrm{if}~r.\gamma=0,$$
$$Res^*(n,j)(r+\gamma)=Res(n,j+1)(r)~\mathrm{if}~r.\gamma=\gamma,$$
$$Res^*(n,j)(r+\gamma)=Res(n,j)(r)~\mathrm{if}~r.\gamma=0.$$

We now prove that it is not possible to define $Res(p,s)$ from the predicates $Fin, C_k,~k\geq 1$ and $Res(q^l,r)$ for primes $q\neq p$.
If there was such a definition, the defining formula would involve only finitely many predicates 
$$Res(q_1^{l_1},{r_1}),\dots,Res(q_m^{l_m},r_m),$$
with each $q_j\neq p$. 
Let $\pi=\prod_{1\leq j\leq m} q_j^{l_j}$. 

By the Chinese remainder theorem there is $N\in \Z$ satisfying the congruences 
$$N \equiv 0~(\mathrm{mod}~\pi)$$
$$N \equiv 1~(\mathrm{mod}~p).$$
Applying the above method we get an element $b\in D$ such that 
$$R\models Res(p,s)(b),$$ 
and new atoms $\gamma_1\dots,\gamma_N$ 
such that 
$$\gamma_i.r=\gamma_i,~\mathrm{if}~r\in D$$
$$\gamma_i.r=0,~\mathrm{if}~r\notin D.$$
The first congruence implies
$$R[\gamma_1,\dots,\gamma_N]\models Res(q_j^{l_j},r_j)(b),$$
and the second congruence 
shows that $Res(p,s)(b)$ does not hold in $R[\gamma_1,\dots,\gamma_N]$. The proof is complete.\end{proof}

\section{Connection to the model theory of restricted products and the ring of adeles of a number field}

By the Feferman-Vaught Theorems \cite{FV} and related results in \cite{DM1}, 
enrichments of infinite atomic Boolean algebras are relevant to the elementary theory of restricted products and adeles. 

Given a language $\mathcal{L}$, an $\mathcal{L}$-formula $\psi(x)$, and $\mathcal{L}$-structures $\mathcal{M}_i$ ($i\in I$), 
the restricted product of $\mathcal{M}_i$ with respect to $\psi(x)$, denoted $\prod^{(\psi)}_{i\in I} \mathcal{M}_i$, 
is defined as the set of all $(a_i)\in \prod_{i\in I} \mathcal{M}_i$ such that 
$$\mathcal{M}_i\models \psi(a_i)~\mathrm{for~almost~all}~i\in I.$$  
Given an $\mathcal{L}$-formula $\Phi(x_1,\dots,x_n)$, and $f_1,\dots,f_n\in \prod_{i\in I} \mathcal{M}_i$, 
the Boolean value is defined by
$$[[\Phi(f_1,\dots,f_n)]]=\{i\in I: \mathcal{M}_i\models \Phi(f_1(i),\dots,f_n(i))\}.$$

Given an extension $\mathcal{L}_{\mathcal{B}}$ of the language of Boolean 
algebras, an $\mathcal{L}_{\mathcal{B}}$-formula $\Psi(z_1,\dots,z_m)$, and $\mathcal{L}$-formulas 
$$\Phi_1(x_1,\dots,x_n),\dots,\Phi_m(x_1,\dots,x_n),$$
enrich the restricted product $\prod^{(\psi)}_{i\in I} \mathcal{M}_i$ by $n$-place relations defined by 
$$Powerset(I)^+\models \Psi([[\Phi_1(\bar x)]],\dots,[[\Phi_m(\bar x)]]),$$
where $Powerset(I)^+$ is the enrichment of $Powerset(I)$ to an $\mathcal{L}_{\mathcal{B}}$-structure. 
These $n$-place relations, for all $n$, yield a language for the restricted product. 

By \cite{FV,DM1}, generalized products  
have quantifier elimination in this language. This reduces the study of the elementary theory of 
$\prod^{(\psi)}_{i\in I} \mathcal{M}_i$ and its definable subsets to that of the enriched 
Boolean algebra $Powerset(I)^+$ and the factors $\mathcal{M}_i$. 

The choice of enrichment of the Boolean algebra $Powerset(I)$ has thus applications 
to the elementary theory and study of definable subsets of 
restricted products.

In this way, the enrichments of $Powerset(I)$ by the predicates in Examples 1,2, and 3, and 
Theorems \ref{thm-1},\ref{thm-2}, and 
\ref{thm-3} allow us to get decidability and quantifier elimination for the ring of adeles of a number field in 
a language stronger than the language of rings, relevant to such matters as the 
product formula for Hilbert symbol (cf.~\cite{DM2}).

%\bibliographystyle{acm}
   %\bibliography{bibdm}

\end{document}